\documentclass[11pt, reqno]{amsart}
\numberwithin{equation}{section}
\usepackage{dsfont}
\usepackage{amscd,amsmath,amssymb,amsfonts}
\usepackage{tikz,color}
\usepackage{hyperref}
\usepackage{enumerate}
\usepackage{latexsym}
\usepackage{cases}
\usepackage{amsthm}
\usepackage{graphicx}
\usepackage{verbatim}
\usepackage{mathrsfs}
\usepackage{geometry}
\newcommand{\R}{\mathbb{R}}

\newtheorem{theorem}{Theorem}[section]

\newtheorem{lemma}[theorem]{Lemma}

\newtheorem{conjecture}[theorem]{Conjecture}
\newtheorem{proposition}[theorem]{Proposition}

\theoremstyle{definition}
\newtheorem{definition}[theorem]{Definition}
\newtheorem{remark}[theorem]{Remark}

\makeatletter
\newcommand{\Extend}[5]{\ext@arrow0099{\arrowfill@#1#2#3}{#4}{#5}}
\makeatother

\begin{document}
\title[Restriction estimate]{Restriction estimates for hyperbolic paraboloids in higher dimensions}

\author[Z. Li]{Zhuoran Li}
	\address{Department of Mathematics, Taizhou University, Taizhou, 225300, China}
	\email{lizhuoran18@gscaep.ac.cn}

\begin{abstract}
In this paper, we prove restriction estimates for hyperbolic paraboloids in dimensions $n \geq 5$ by the polynomial partitioning method.
\end{abstract}

\maketitle

\begin{center}
\begin{minipage}{120mm}
\end{minipage}
\end{center}

%%%%%%%%%%%%%%%%%%%%%%%%%%%%%%%%%%%%%%%%%%%%%%%%%%%%%%%%%%%%%%%%%%%%%%%%%%%%%%%%%%%%%%%%%%%%%%%%%%%%%%%%%%%%%%%%%%%%%%%%%%%%%%%%%%%%%%%%%%

%%%%%%%%%%%%%%%%%%%%%%%%%%%%%%%%%%%%%%%%%%%%%%%%%%%%%%%%%%  section 1  %%%%%%%%%%%%%%%%%%%%%%%%%%%%%%%%%%%%%%%%%%%%%%%%%%%%%%%%%%

%%%%%%%%%%%%%%%%%%%%%%%%%%%%%%%%%%%%%%%%%%%%%%%%%%%%%%%%%%%%%%%%%%%%%%%%%%%%%%%%%%%%%%%%%%%%%%%%%%%%%%%%%%%%%%%%%%%%%%%%%%%%%%%%%%%%%%%%%%

\section{Introduction and main result}

\noindent

Let $B^{n-1}(0,1)$ be the unit ball centered at the origin in $\mathbb{R}^{n-1}$. We define the Fourier extension operator
\begin{equation}\label{extensionoperator}
E_Mf(x):= \int_{B^{n-1}(0,1)}f(\xi)e[x_1\xi_1+\cdot\cdot\cdot+x_{n-1}\xi_{n-1}+x_n(M\xi \cdot \xi)]d\xi,
\end{equation}
where $x=(x_1,\cdots,x_n)\in\R^n,\;\xi=(\xi_1,\cdot\cdot\cdot,\xi_{n-1})\in \mathbb{R}^{n-1},\;e(t):= e^{it}$ for $t\in \mathbb{R}$ and
\begin{equation*}
M=\left(
  \begin{array}{cc}
    I_{n-1-m} & O \\
    O & -I_{m} \\
  \end{array}
\right).
\end{equation*}
We may assume $m = \min \{n-1-m, m\}$ since otherwise we can replace $x_n$ by $-x_n$.

E. M. Stein \cite{Stein} proposed the restriction conjecture in the 1960s. The adjoint form of the restriction conjecture can be stated as follows.
\begin{conjecture}\label{adjointconj}
For any $p>\frac{2n}{n-1}$ and $p\geq \frac{n+1}{n-1}q'$, there holds
\begin{equation}\label{conj:global}
\Vert E_Mf \Vert_{L^p(\mathbb{R}^n)}\leq C_{p,q}\Vert f \Vert_{L^q(B^{n-1}(0,1))}.
\end{equation}
\end{conjecture}

The conjecture in $\mathbb{R}^2$ was proved by Fefferman \cite{F70} and Zygmund \cite{Z74} independently. In $\mathbb{R}^n\;(n\geq 3)$ the conjecture is open. In the case $m=0$, we refer to \cite{Bo91, Tao03, BG, Guth1, Guth2, Wang, D19, HR, HZ, WW22, GuoWangZhang22} for some partial progress. For $1\leq m \leq \frac{n-1}{2}$, we refer to \cite{Vargas05, Lee06, CL17, Stova19, HI, Barron, BMV20a, GO20, LiZheng23} and the references therein.

By the standard $\varepsilon$-removal argument in \cite{BG, Tao99}, Conjecture \ref{adjointconj} can be reduced to the following local version.
\begin{conjecture}[Local restriction conjecture]\label{conj:local}
Let $p\geq \frac{2n}{n-1}$ and $p \geq \frac{n+1}{n-1}q'$. Then, for any $\varepsilon>0$, there exists a positive constant $C(\varepsilon)$ such that
\begin{equation}\label{equ:parloc}
  \| E_Mf \|_{L^p(B_R)}\leq C(\varepsilon)R^{\varepsilon}\|f\|_{L^q(B^{n-1}(0,1))},
\end{equation}
where $B_R$ denotes an arbitrary ball of radius $R$ in $\mathbb{R}^n$.
\end{conjecture}

In this paper, we study the restriction problem associated to the operator $E_{M}$. Our main result is the following.
\begin{theorem}\label{mainprop0}
Let $n \geq 5$. The estimate
\begin{equation}\label{equ:goalformprop0}
\Vert E_{M}f \Vert_{L^p(\mathbb{R}^n)}\leq C_p \Vert f \Vert_{L^{p}(B^{n-1}(0,1))}
\end{equation}
holds for $p>\max\{\frac{2(n^2+2n-2)}{n^2-1}, \frac{2(n-m)}{n-m-1}\}$.
\end{theorem}

\begin{remark}
Hickman-Iliopoulou \cite{HI} proved sharp $L^p$ estimates for H\"{o}rmander-type oscillatory integral operators of arbitrary signature, which implies that
\[\Vert E_{M}f \Vert_{L^p(\mathbb{R}^n)}\leq C_p \Vert f \Vert_{L^{p}(B^{n-1}(0,1))}\]
holds whenever $p$ satisfies
\begin{align*}
\begin{split}
p>\left\{
\begin{array}{ll}
\frac{2(3n-2m+1)}{3n-2m-3}, & \text{if}\;n\;\text{is}\;\text{odd},\\
\frac{2(3n-2m+2)}{3n-2m-2}, & \text{if}\;n\;\text{is}\;\text{even}.
\end{array}
\right.
\end{split}
\end{align*}
Theorem \ref{mainprop0} improves the above result of Hickman-Iliopoulou when $n\geq 5$ is odd and $m=\frac{n-3}{2}$.
\end{remark}

By the standard $\varepsilon$-removal argument in \cite{Tao99}, Theorem \ref{mainprop0} reduces to the following local version.
\begin{theorem}\label{localmainprop0}
Let $n \geq 5$. For every $\varepsilon >0$, there exists a positive constant $C_{\varepsilon}$ such that for any sufficiently large $R$
\begin{equation}\label{eq:localmainprop0}
\Vert E_Mf \Vert_{L^p(B_R)}\leq C_{\varepsilon}R^{\varepsilon}\Vert f \Vert_{L^{p}(B^{n-1}(0,1))}
\end{equation}
holds for $p\geq \max\{\frac{2(n^2+2n-2)}{n^2-1}, \frac{2(n-m)}{n-m-1}\}$ and any function $f\in L^{p}(B^{n-1}(0,1))$, where $B_R$ denotes an arbitrary ball of radius $R$ in $\mathbb{R}^n$.
\end{theorem}

\vskip 0.2in

The paper is organized as follows. In Section 2, we review some basic tools. In Section 3, we discuss the broad norm adapted to hyperbolic paraboloids. In Section 4, we give the proof of Theorem \ref{localmainprop0}.

\vskip 0.2in

{\bf Notations:} For nonnegative quantities $X$ and $Y$, we will write $X\lesssim Y$ to denote the estimate $X\leq C Y$ for some large constant $C$ which may vary from line to line and depend on various parameters. If $X\lesssim Y\lesssim X$, we simply write $X\sim Y$. Dependence of implicit constants on the power $p$ or the dimension will be suppressed; dependence on additional parameters will be indicated by subscripts.
For example, $X\lesssim_u Y$ indicates $X\leq CY$ for some $C=C(u)$. We write $B^{d}(0,r)$ to denote the ball centered at the origin of radius $r$ in $\mathbb{R}^d$. We denote an arbitrary ball of radius $R$ in $\mathbb{R}^n$ by $B_R$. For any set $F \subset \mathbb{R}^d$, we denote the characteristic function on $F$ by $\chi_{F}$. Usually, the Fourier transform on $\mathbb{R}^d$ is defined by
\begin{equation*}
\widehat{f}(\xi):=(2\pi)^{-d}\int_{\mathbb{R}^d}e^{-ix\cdot \xi}f(x)\;dx.
\end{equation*}

\section{Basic tools}

In this section, we review the basic tools that we will need in the proof of Theorem \ref{localmainprop0}.

\subsection{Bilinear estimates}

We review some bilinear restriction estimates and a geometric lemma that characterizes what happens if these bilinear estimates fail. The following estimate was proved by S. Lee \cite{Lee06} in dimensions $n \geq 3$ and independently proved by Vargas \cite{Vargas05} in $\mathbb{R}^3$.

\begin{theorem}[\cite{Lee06, Vargas05}]\label{thm:bilinear0}  Suppose $f_1$ and $f_2$ are supported in open balls $\tau_1$ and $\tau_2$ of radius $\sim 1$. If
\begin{equation}\label{eq:bilinearCondition0}
\inf_{ \substack{ \xi, \bar{\xi} \in \tau_1 \\ \eta, \bar{\eta} \in \tau_2 } } | M(\xi - \eta) \cdot (\bar{\xi} - \bar{\eta})| \geq c > 0
\end{equation}
then
\begin{equation}\label{eq:bilnearEst0}
\| |E_Mf_1 E_Mf_{2}|^{\frac{1}{2}} \|_{L^{p}(B^n_R)} \leq C_{\epsilon}R^{\epsilon} \|f_1\|^{\frac{1}{2}}_{L^2} \|f_2\|_{L^2}^{\frac{1}{2}} \end{equation}
holds whenever $p \geq \frac{2(n+2)}{n}$.
\end{theorem}

We recall a geometric lemma which was proved by Barron \cite{Barron}. To state it, we need the following notation. Let $K:=R^{\varepsilon^2}$ be a large number. Suppose that $\tau_1$ and $\tau_2$ are two $K^{-1}$-balls in $B^{n-1}(0,1)$ whose centers are separated by $\sim K^{-1}$. We say that two $K^{-1}$-caps $\tau_1, \tau_2$ are strongly separated if
\begin{equation}\label{eq:bilinearCondition}
\inf_{ \substack{ \xi, \bar{\xi} \in \tau_1 \\ \eta, \bar{\eta} \in \tau_2 } } | M(\xi - \eta) \cdot (\bar{\xi} - \bar{\eta})| \gtrsim K^{-1}
\end{equation}
holds.

\begin{lemma}[\cite{Barron}]\label{lem:geometricLemma}
Let $\{\tau\}$ be a collection of finitely-overlapping $K^{-1}$-balls in $B^{n-1}(0,1)$ with $E_Mf = \sum_{\tau} E_Mf_{\tau}$. Recall that
$\min \{n-1-m,m\}=m$. If $K$ is sufficiently large then one of the following must occur.
\begin{enumerate}
		\item There exists a uniform constant $\epsilon_0 \in (0,\frac{1}{4(n-1)})$ and an $m$-dimensional affine space $V$ in $\mathbb{R}^{n-1}$ such that every $\tau$ is contained in an $O(K^{-\epsilon_0})$ neighborhood of $V$.
		\item There are two $K^{-1}$-balls $\tau_1, \tau_2$ for which
$$\inf_{ \substack{ \xi, \bar{\xi} \in \tau_1 \\ \eta, \bar{\eta} \in \tau_2 } } | M(\xi - \eta) \cdot (\bar{\xi} - \bar{\eta})| \gtrsim K^{-1}.$$
\end{enumerate}
\end{lemma}

The following version of Theorem \ref{thm:bilinear0} is to be used in the proof of Theorem \ref{localmainprop0}.
\begin{theorem}[\cite{Lee06, Vargas05}]\label{thm:usefulbilinear0} Suppose $f_1$ and $f_2$ are supported in open balls $\tau_1$ and $\tau_2$ of radius $K^{-1}$. If $\tau_1, \tau_2$ satisfy the condition \eqref{eq:bilinearCondition}, then
\begin{equation}\label{eq:bilnearEst0}
\| |E_Mf_1 E_Mf_{2}|^{\frac{1}{2}} \|_{L^{p}(B^n_R)} \leq C_{\epsilon}K^{O(1)}R^{\epsilon} \|f_1\|^{\frac{1}{2}}_{L^2} \|f_2\|_{L^2}^{\frac{1}{2}} \end{equation}
holds whenever $p \geq \frac{2(n+2)}{n}$.
\end{theorem}
We refer to \cite{Barron} for the details of Theorem \ref{thm:usefulbilinear0}.

\subsection{Wave packet decomposition}

We recall the wave packet decomposition at scale $R$ following the description of \cite{Guth2, Wang}. We decompose the unit ball in $\mathbb{R}^{n-1}$ into finitely overlapping small balls $\theta$ of radius $R^{-1/2}$. These small disks are referred to as $R^{-1/2}$-caps. Let $\{\psi_{\theta}\}$ be a smooth partition of unity adapted to $\{\theta\}$. We write
\[f=\sum_{\theta}\psi_{\theta}f,\]
and define
\[f_{\theta}:=\psi_{\theta}f.\]
We cover $\mathbb{R}^{n-1}$ by finitely overlapping balls of radius about $R^{\frac{1+\delta}{2}}$, centered at vectors $\upsilon\in R^{\frac{1+\delta}{2}}\mathbb{Z}^{n-1}$, where $\delta$ is a small number satisfying $\delta \sim \varepsilon^3$. Let $\{\eta_{\upsilon}\}$ be a smooth partition of unity adapted to this cover. We can now decompose
\[f=\sum_{\theta,\upsilon}\big(\eta_{\upsilon}(\psi_{\theta}f)^{\wedge}\big)^{\vee}=\sum_{\theta,\upsilon}\eta_{\upsilon}^{\vee}\ast(\psi_{\theta}f).\]
We choose smooth functions $\tilde{\psi}_{\theta}$ such that $\tilde{\psi}_{\theta}$ is supported on $\theta$ but $\tilde{\psi}_{\theta}=1$ on a $cR^{-1/2}$ neighborhood of the support of $\psi_{\theta}$ for a small constant $c>0$. We define
\[f_{\theta,\upsilon}:=\tilde{\psi}_{\theta}[\eta_{\upsilon}^{\vee}\ast(\psi_{\theta}f)].\]
Since $\eta_{\upsilon}^{\vee}(x)$ is rapidly decaying for $\vert x \vert\gtrsim R^{\frac{1-\delta}{2}}$,
\[f=\sum_{(\theta,\upsilon):d(\theta)=R^{-1/2}}f_{\theta,\upsilon}+RapDec(R)\Vert f \Vert_{L^2},\]
where $d(\theta)$ denotes the diameter of $\theta$, and $RapDec(R)$ means that the quantity is bounded by $O_N(R^{-N})$ for any large integer $N>0$.

The wave packets $E_Mf_{\theta,\upsilon}$ satisfy two useful properties. The first property is that the functions $f_{\theta,\upsilon}$ are approximately orthogonal. The second property is that on the ball $B_R$ the function $E_Mf_{\theta,\upsilon}$ is essentially supported on the tube
\[T_{\theta,\upsilon}:=\{(x',x_n)\in B^n(0,R), \vert x'+2x_n\omega_{\theta}+\upsilon \vert\leq R^{1/2+\delta}\},\]
where $\omega_{\theta}$ denotes the center of $\theta$.

Let $\mathbb{T}$ be a set of $(\theta,\upsilon)$. We say that $f$ is concentrated on wave packets from $\mathbb{T}$ if \[f=\sum_{(\theta,\upsilon)\in\mathbb{T}}f_{\theta,\upsilon}+RapDec(R)\Vert f \Vert_{L^2}.\]

\subsection{Flat decoupling}

We will employ the following result.
\begin{proposition}[Flat Decoupling] \label{prop:decoupling}
Suppose $\mathcal{T}$ is a collection of finitely-overlapping $\rho$-caps $\tau$ with $\rho^{-1} < R$. Then
$$\| \sum_{\tau \in \mathcal{T}} E_Mf_{\tau}\|_{L^{p}(B_R)} \lesssim (\# \mathcal{T})^{\frac{1}{2}-\frac{1}{p}}(\log \rho)^{O(1)}\big(\sum_{\tau \in \mathcal{T}} \|E_Mf_{\tau}\|^{2}_{L^{p}(w_{B_R})}\big)^{\frac{1}{2}},$$
where $w_{B_R}$ is a smooth weight function adapted to $B_R$.
\end{proposition}

\begin{proof} The case $p = \infty$ is just the Cauchy-Schwarz inequality, and when $p =2$ the proposition follows from Plancharel's theorem. The remaining cases follow from the interpolation argument in \cite{BD15}.
\end{proof}

Finally we recall that if $R$ is small enough then Theorem \ref{localmainprop0} follows directly from H\"{o}lder's inequality. We can therefore assume by induction that Theorem \ref{localmainprop0} is true at scale $r$ whenever $r \leq \frac{R}{2}$. For technical reasons related to the decoupling result in Proposition \ref{prop:decoupling} we can also assume that the following weighted estimate holds. For every $\varepsilon> 0$, there exists a positive constant $C_{\varepsilon}$ such that
$$\|E_Mf\|_{L^p (w_{B_r})} \leq C_{\epsilon}r^{\epsilon}\|f\|_{L^p(B^{n-1}(0,1))}$$
holds for any $r \leq \frac{R}{2}$ and $p\geq \max\{\frac{2(n^2+2n-2)}{n^2-1}, \frac{2(n-m)}{n-m-1}\}$, where $w_{B_r}$ is a smooth weight function adapted to $B_r$.

\section{\textbf{Broad norm adapted to hyperbolic paraboloids}}

Basic on the definition of the $k$-broad norm introduced by Guth \cite{Guth2}, we discuss the broad norm adapted to hyperbolic paraboloids.

We decompose $B^{n-1}(0,1)$ into finitely overlapping $K^{-1}$-balls $\tau$ with $K= R^{\varepsilon^2}$. We write $f=\sum_{\tau}f_{\tau}$ by partition of unity with $f_{\tau}$ supported in $\tau$. We use a collection $\mathcal{B}$ of finitely-overlapping $K$-balls to cover $B_R$ on the spatial side. Let $K_1:=R^{\varepsilon^4}$. For any integer $A\sim R^{\varepsilon^6}$, we define
\begin{equation}\label{pmass}
 \mu_{E_Mf}(B):= \min_{V_1,...,V_{A}}\max_{\tau \notin V_s\;\text{for\;all}\;1\leq s \leq A}\Vert E_Mf_{\tau} \Vert^p_{L^p(B)}
\end{equation}
for every ball $B\in \mathcal{B}$, where each $V_s$ denotes an $m$-dimensional affine space in $\mathbb{R}^{n-1}$ and $\tau \notin V_s$ means $dist(\tau, V_s)>K_1^{-1}$. If a set $U$ is (finitely overlapping) tiled by $B$, then we define $\mu_{E_Mf}(U)$ by summing over the tiling
\[\mu_{E_Mf}(U):= \sum_{B \subset U} \mu_{E_Mf}(B).\]
We define the broad norm adapted to hyperbolic paraboloids on $B_R$ as
\[\Vert E_Mf \Vert^p_{HL^p_{A}(B_R)}:=\sum_{B \subset B_R}\mu_{E_Mf}(B).\]

By the uncertainty principle, $\vert E_Mf_{\tau} \vert$ is essentially a constant on any $B\in \mathcal{B}$. The usual triangle inequality and H\"{o}lder's inequality for the $HL^p_{A}$-norm can be verified with a modification of $A$.

\begin{lemma}\label{hbtriangle}
Suppose that $f=g+h$ and suppose that $A=A_1+A_2$, where $A, A_1, A_2$ are non-negative integers. Then
\[\Vert E_Mf \Vert_{HL^p_{A}(U)}\lesssim \Vert E_Mg \Vert_{HL^p_{A_1}(U)}+\Vert E_Mh \Vert_{HL^p_{A_2}(U)}.\]
\end{lemma}

\begin{lemma}\label{hbHolder}
Suppose that $1\leq p, p_1, p_2< \infty$ and $0\leq \alpha_1, \alpha_2 \leq 1$ obey $\alpha_1+\alpha_2=1$ and
\[\frac{\alpha}{p}=\frac{\alpha_1}{p_1}+\frac{\alpha_2}{p_2}.\]
Also, suppose that $A=A_1+A_2$. Then
\[\Vert E_Mf \Vert_{HL^p_{A}(U)}\leq \Vert E_Mf \Vert^{\alpha_1}_{HL^{p_1}_{A_1}(U)}\Vert E_Mf \Vert^{\alpha_2}_{HL^{p_2}_{A_2}(U)}.\]
\end{lemma}

The value of $A$ can change from line to line whenever we apply Lemma \ref{hbtriangle} and Lemma \ref{hbHolder}. In what follows, we omit $A$ from our notation unless it plays a role in the proof of Theorem \ref{localmainprop0}.

\section{The proof of Theorem \ref{localmainprop0}}

In this section, we prove Theorem \ref{localmainprop0}. Let $p_{broad}:=\frac{2(n^2+2n-2)}{n^2-1}$ and $p_{narrow}:=\frac{2(n-m)}{n-m-1}$.
We denote $\max\{p_{broad}, p_{narrow}\}$ by $\bar{p}$.

Given $B \in \mathcal{B}$, we define its significant set
$$\mathcal{S}_{p}(B):= \{ \tau: \|E_Mf_{\tau} \|_{L^p(B)} \geq \frac{1}{K^{n-1+\varepsilon}}\|E_Mf\|_{HL^p(B)}\},$$
and write
\[\|E_Mf\|_{HL^{p}(B)}\leq \|\sum_{\tau \in \mathcal{S}_{p}(B)}E_Mf_{\tau}\|_{HL^{p}(B)}+\|\sum_{\tau \notin \mathcal{S}_{p}(B)}E_Mf_{\tau}\|_{HL^{p}(B)}.\]
It is easy to see
$$\|\sum_{\tau \notin \mathcal{S}_{p}(B)} E_Mf_{\tau}\|_{HL^{p}(B)}\lesssim  K^{n-1}\max_{\tau \notin \mathcal{S}_p(B)}
\Vert E_Mf_{\tau} \Vert_{L^p(B)}\lesssim K^{-\varepsilon}\|E_Mf\|_{HL^{p}(B)}.$$
The following lemma is useful.
\begin{lemma}\label{hypervsbil}
For any function $g$ and sufficiently large number $R$, there holds
\begin{equation}\label{eq:hypervsbil}
\Vert E_Mg \Vert^p_{HL^p(B_R)}\lesssim K^{O(1)}\sum_{\tau_1,\tau_2\;\text{are\;strongly\;separated}}\| |E_Mg_{\tau_1} E_Mg_{\tau_2}|^{\frac{1}{2}} \|^p_{L^{p}(B_R)}.
\end{equation}
\end{lemma}

\begin{proof}
For any given $B\in \mathcal{B}$, if there are two strongly separated $K^{-1}$-balls $\tau_1,\tau_2$ in $\mathcal{S}_{p}(B)$ then
\begin{align*}
\Vert E_Mg \Vert^p_{HL^p(B)}\lesssim& K^{O(1)}\|E_Mg_{\tau_1}\|^{p/2}_{L^p(B)}\|E_Mg_{\tau_2}\|^{p/2}_{L^p(B)} \\
\sim& K^{O(1)}\sum_{\tau_1,\tau_2\;\text{are\;strongly\;separated}}\| |E_Mg_{\tau_1} E_Mg_{\tau_2}|^{\frac{1}{2}}\|^p_{L^{p}(B)}.
\end{align*}
Here we used the fact that $\vert E_Mg_{\tau_i}\vert\;(i=1,2)$ is essentially a constant on any given $B\in \mathcal{B}$. Otherwise, by Lemma \ref{lem:geometricLemma}, there exists an $m$-dimensional affine space $V$ in $\mathbb{R}^{n-1}$ such that $dist(\tau, V)\leq K_1^{-1}$ for all $\tau \in \mathcal{S}_p(B)$. It follows
\begin{equation*}
\begin{split}
  \Vert E_Mg \Vert^p_{HL^p(B)}\leq& \min_{V_1=V,...,V_{A}}\max_{\tau \notin V_s\;\text{for\;all}\;1\leq s \leq A}\Vert E_Mg_{\tau} \Vert^p_{L^p(B)} \\
                              \leq& \max_{\tau \notin \mathcal{S}_p(B)}\Vert E_Mg_{\tau} \Vert^p_{L^p(B)} \\
                              \leq& K^{-p(n-1+\varepsilon)}\Vert E_Mg\Vert^p_{HL^p(B)}.
\end{split}
\end{equation*}
Thus, we obtain
\begin{equation*}
\begin{split}
  \Vert E_Mg \Vert^p_{HL^p(B)}\lesssim& K^{O(1)}\sum_{\tau_1,\tau_2\;\text{are\;strongly\;separated}}\| |E_Mg_{\tau_1} E_Mg_{\tau_2}|^{\frac{1}{2}} \|^p_{L^{p}(B)} \\
  +& K^{-p(n-1+\varepsilon)}\Vert E_Mg\Vert^p_{HL^p(B)}
\end{split}
\end{equation*}
for any $B\in \mathcal{B}$. Consequently,
\[\Vert E_Mg \Vert^p_{HL^p(B)}\lesssim K^{O(1)}\sum_{\tau_1,\tau_2\;\text{are\;strongly\;separated}}\| |E_Mg_{\tau_1} E_Mg_{\tau_2}|^{\frac{1}{2}} \|^p_{L^{p}(B)}\]
holds for any $B\in \mathcal{B}$. Summing over all the balls $B\in \mathcal{B}$, we get
\[\Vert E_Mg \Vert^p_{HL^p(B_R)}\lesssim K^{O(1)}\sum_{\tau_1,\tau_2\;\text{are\;strongly\;separated}}\| |E_Mg_{\tau_1} E_Mg_{\tau_2}|^{\frac{1}{2}} \|^p_{L^{p}(B_R)}\]
as desired.
\end{proof}

We begin to estimate $\Vert E_Mf \Vert^p_{HL^p(B_R)}$.
\begin{proposition}\label{strongermainprop0}
For every $\varepsilon >0$, there exists a positive constant $C_{\varepsilon}$ such that for any sufficiently large $R$
\[\Vert E_Mf \Vert^p_{HL^p(B_R)}\leq C_{\varepsilon}^p R^{p\varepsilon}\Vert f \Vert^{\frac{2n}{n-1}}_{L^2(B^{n-1}(0,1))}
\max_{d(\theta)=R^{-1/2}}\Vert f_{\theta} \Vert^{p-\frac{2n}{n-1}}_{L^{2}_{avg}(\theta)}\]
holds for $p\geq p_{broad}$ and any function $f\in L^p(B^{n-1}(0,1))$, where
\[\Vert f_{\theta} \Vert_{L^{2}_{avg}(\theta)}:=\Big(\frac{1}{\vert \theta \vert}\Vert f_{\theta} \Vert^2_{L^2(\theta)}\Big)^{1/2}\]
and $\vert \theta \vert$ denotes the Lebesgue measure of $\theta$.
\end{proposition}

\begin{proof}
We apply the polynomial partitioning technique of Guth \cite{Guth1} to the broad norm of $\vert \chi_{B_R}E_Mf \vert$. By Theorem 0.6 in \cite{Guth1}, for each degree $D\sim \log R$, we can find a nonzero polynomial $P$ of degree at most $D$ so that the complement of its zero set $Z(P)$ in $B_R$ is a union of $O(D^n)$ disjoint cells $U'_i$:
\[B_R \setminus Z(P)= \bigcup U'_i,\]
and the broad norm is roughly the same in each cell
\[\Vert E_Mf \Vert^p_{HL^p(U'_i)}\sim D^{-n}\Vert E_Mf \Vert^p_{HL^p(B_R)}.\]
The cells $U'_i$'s might have various shape. For the purpose of induction on scales, we would like to put it inside a smaller ball of radius $\frac{R}{D}$. To do so, it suffices to multiply $P$ by another polynomial $G$ of degree $nD$, and consider the cells cut-off by the zero set of $P\cdot G$. More precisely, let $G_j, j=1,...,n$ be the product of polynomials of degree $1$ whose zero sets are the hyperplanes parallel to the $x_j$-axis, of spacing $\frac{R}{D}$ and intersecting $B_R$. The degree of $G_j$ is at most $D$. Denote $\prod^n_{j=1}G_j$ by $G$. Let $Q=P\cdot G$ be the new partitioning polynomial, then we have a new decomposition of $B_R$:
\[B_R \setminus Z(Q)=\bigcup O'_i.\]
The zero set $Z(Q)$ decomposes $B_R$ into at most $O(D^n)$ cells $O'_i$ by the Milnor-Thom theorem \cite{Milnor1964, Thom1965}. A wave packet $E_Mf_{\theta,\upsilon}$ has negligible contribution to a cell $O'_i$ if its essential support $T_{\theta,\upsilon}$ does not intersect $O'_i$. To analyze how $T_{\theta,\upsilon}$ intersects a cell $O'_i$, we need to shrink $O'_i$ further. We define the wall $W$ as the $R^{\frac{1}{2}+\delta}$-neighborhood of $Z(Q)$ in $B_R$ and define the new cells as
\[O_i:=O'_i \setminus W.\]

In summary, we get
\[B_R = W \cup (\bigcup O_i),\]
\begin{equation}\label{decomposition}
\Vert E_Mf \Vert^p_{HL^p(B_R)}=\Vert E_Mf \Vert^p_{HL^p(W)}+\sum_{i}\Vert E_Mf \Vert^p_{HL^p(O_i)}
\end{equation}
and
\begin{equation}\label{equalcell}
\Vert E_Mf \Vert^p_{HL^p(O_i)}\lesssim D^{-n}\Vert E_Mf \Vert^p_{HL^p(B_R)}.
\end{equation}

We are in the cellular case if
\[\Vert E_Mf \Vert^p_{HL^p(B_R)}\lesssim \sum_{i}\Vert E_Mf \Vert^p_{HL^p(O_i)}.\]
We define
\[f_i:=\sum_{T_{\theta,\upsilon}\cap O_i \neq \emptyset} f_{\theta,\upsilon}.\]
Since the wave packets $E_Mf_{\theta,\upsilon}$ with $T_{\theta,\upsilon}\cap O_i=\emptyset$ have negligible contribution to
$\Vert E_Mf \Vert_{HL^p(O_i)}$,
\[\Vert E_Mf \Vert_{HL^p(O_i)}=\Vert E_Mf_i \Vert_{HL^p(O_i)}+RapDec(R)\Vert f \Vert_{L^2}.\]
Each tube $T_{\theta,\upsilon}$ intersects at most $D+1$ cells $O_i$. Consequently,
\[\sum_i \Vert f_i \Vert^2_{L^2}\lesssim D \Vert f \Vert^2_{L^2}.\]
By \eqref{equalcell} and the definition of the cellular case, there are at least $O(D^n)$ cells $O_i$ such that
\begin{equation}\label{icelldecomposition}
\Vert E_Mf \Vert^p_{HL^p(B_R)}\lesssim D^n \Vert E_Mf_i \Vert^p_{HL^p(O_i)}.
\end{equation}
Since there are $O(D^n)$ cells,
\begin{equation}\label{icell2}
\Vert f_i \Vert_{L^2}\lesssim D^{-\frac{n-1}{2}} \Vert f \Vert_{L^2}
\end{equation}
holds for most of the cells. If we are in the cellular case, we derive by induction on scales and \eqref{icell2} that
\begin{align*}
\Vert E_Mf \Vert^p_{HL^p(B_R)}\lesssim& D^n \Vert E_Mf_i \Vert^p_{HL^p(O_i)}\\
\lesssim& C^p_{\varepsilon}\big(\frac{R}{D}\big)^{p\varepsilon}D^n\Vert f_i \Vert^{\frac{2n}{n-1}}_{L^2}\max_{d(\tau)=(\frac{R}{D})^{-1/2}}\Vert f_{i,\tau} \Vert^{p-\frac{2n}{n-1}}_{L^{2}_{avg}(\tau)}\\
\lesssim& C^p_{\varepsilon}R^{p\varepsilon}D^{-p\varepsilon}\Vert f \Vert^{\frac{2n}{n-1}}_{L^2}\max_{d(\theta)=R^{-1/2}}\Vert f_{\theta} \Vert^{p-\frac{2n}{n-1}}_{L^{2}_{avg}(\theta)}.
\end{align*}
Since $D\sim \log R$, we take $R$ to be sufficiently large so that the induction closes.

If we are not in the cellular case, then
\[\Vert E_Mf \Vert_{HL^p(B_R)}\lesssim \Vert E_Mf \Vert_{HL^p(W)}.\]
We call it the algebraic case because the broad norm is concentrate on the neighborhood of an algebraic surface. Only the wave packets $E_Mf_{\theta,\upsilon}$ whose essential supports $T_{\theta,\upsilon}$ intersect $W$ contribute to $\Vert E_Mf \Vert_{HL^p(W)}$. Depending on how they intersect, we identify a tangential part, which consists of the wave packets tangential to $W$, and a transversal part, which consists of the wave packets intersecting $W$ transversely. We recall the definitions of the tangential tubes and the transversal tubes in \cite{Guth1} here. We cover $W$ by finitely overlapping balls $B_j$ of radius $\rho:=R^{1-\delta}$.

\begin{definition}
$\mathbb{T}_{j,tang}$ is the set of all tubes $T$ obeying the following two conditions.\\
(1)\;$T\cap W \cap B_j \neq \emptyset$.\\
(2)\;If $z$ is any non-singular point of $Z(P)$ lying in $2B_j \cap 10T$, then
\[Angle(\upsilon(T),T_zZ)\leq R^{-1/2+2\delta}.\]
\end{definition}

\begin{definition}
$\mathbb{T}_{j,trans}$ is the set of all tubes $T$ obeying the following two conditions.\\
(1)\;$T\cap W \cap B_j \neq \emptyset$.\\
(2)\;There exists a non-singular point $z$ of $Z(P)$ lying in $2B_j \cap 10T$ such that
\[Angle(\upsilon(T),T_zZ)> R^{-1/2+2\delta}.\]
\end{definition}

We write
\[\Vert E_Mf \Vert^p_{HL^p(W)}\lesssim \sum_{B_j}\Vert E_Mf_{j,tang} \Vert^p_{HL^p(W\cap B_j)}+\sum_{B_j}\Vert E_Mf_{j,trans} \Vert^p_{HL^p(W\cap B_j)}.\]

We are in the transversal case if
\[\Vert E_Mf \Vert^p_{HL^p(B_R)}\lesssim \sum_{B_j} \Vert E_Mf_{j,trans} \Vert^p_{HL^p(W\cap B_j)}.\]
The treatment of the transversal case is similar to the cellular case, which requires the following lemma in place of \eqref{icell2}.

\begin{lemma}[\cite{Guth2}]\label{transgeometriclemma}
Each tube $T$ belongs to at most $Poly(D)$ different sets $\mathbb{T}_{j,trans}$. Here $Poly(D)$ means a quantity bounded by a constant power of $D$.
\end{lemma}
By Lemma \ref{transgeometriclemma}, we have
\begin{equation}\label{transgeoineq}
\sum_{B_j}\Vert f_{j,trans} \Vert^2_{L^2}\lesssim Poly(D)\Vert f \Vert^2_{L^2}.
\end{equation}
If we are in the transversal case, we derive by \eqref{transgeoineq} and induction on scales
\begin{align*}
\Vert E_Mf \Vert^p_{HL^p(B_R)}\lesssim& \sum_{B_j}\Vert E_Mf_{j,trans} \Vert^p_{HL^p(W\cap B_j)}\\
\lesssim& C^p_{\varepsilon}\rho^{p\varepsilon}\sum_{B_j}\Vert f_{j,trans} \Vert^{\frac{2n}{n-1}}_{L^2}\max_{d(\tau)=\rho^{-1/2}}\Vert f_{j,trans,\tau} \Vert^{p-\frac{5}{2}}_{L^{2}_{avg}(\tau)}\\
\lesssim& C^p_{\varepsilon}\rho^{p\varepsilon}\big(\sum_{B_j}\Vert f_{j,trans} \Vert^2_{L^2}\Big)^{\frac{n}{n-1}}\max_{d(\theta)=R^{-1/2}}\Vert f_{\theta} \Vert^{p-\frac{2n}{n-1}}_{L^{2}_{avg}(\theta)}\\
\lesssim& C^p_{\varepsilon}R^{(1-\delta)p\varepsilon}(Poly(D))^{O(1)}\Vert f \Vert^{\frac{2n}{n-1}}_{L^2}\max_{d(\theta)=R^{-1/2}}\Vert f_{\theta} \Vert^{p-\frac{2n}{n-1}}_{L^{2}_{avg}(\theta)}.
\end{align*}
Recall that $D\sim \log R$. The induction closes for sufficiently large $R$.

We are in the tangential case if
\[\Vert E_Mf \Vert^p_{HL^p(B_R)}\lesssim \sum_{B_j} \Vert E_Mf_{j,tang} \Vert^p_{HL^p(W\cap B_j)}.\]
We claim that the following $L^2$ estimate holds.
\begin{equation}\label{L2energy}
\Vert E_Mf_{j,tang} \Vert_{HL^2(W\cap B_j)}\lesssim \rho^{1/2}\Vert f_{j,tang} \Vert_{L^2}.
\end{equation}
In fact, we have
\begin{align*}
\Vert E_Mf_{j,tang}\Vert^2_{HL^2(W\cap B_j)}\lesssim& \sum_{B\subset W\cap B_j}\Vert E_Mf_{j,tang} \Vert^2_{HL^2(B)}\\
\lesssim& \sum_{B\subset W\cap B_j}\max_{\tau}\Vert E_Mf_{j,tang,\tau} \Vert^2_{L^2(B)}\\
\lesssim& \sum_{B\subset W\cap B_j}\sum_{\tau}\Vert E_Mf_{j,tang,\tau} \Vert^2_{L^2(B)}\\
\lesssim& \sum_{\tau}\Vert E_Mf_{j,tang,\tau} \Vert^2_{L^2(W\cap B_j)}\\
\lesssim& \sum_{\tau}\rho \Vert f_{j,tang,\tau} \Vert^2_{L^2}\\
\lesssim& \rho \Vert f_{j,tang} \Vert^2_{L^2}.\\
\end{align*}
This verifies Claim \eqref{L2energy}. The Polynomial Wolff Axioms proved by Katz-Rogers \cite{KR19} say that $\mathrm{supp}\;f_{j,tang}$ lies in $\lesssim R^{\frac{n-2}{2}+O(\delta)}$ caps $\theta$ of radius $R^{-1/2}$. As a result, we have
\begin{equation}\label{keytangentestimate}
\Vert f_{j,tang} \Vert_{L^2}\lesssim R^{-1/4+O(\delta)}\max_{d(\theta)=R^{-1/2}}\Vert f_{\theta} \Vert_{L^{2}_{avg}(\theta)}.
\end{equation}

Applying Lemma \ref{hypervsbil} and Theorem \ref{thm:usefulbilinear0} to the function $E_Mf_{j,tang}$, we derive that
\begin{equation}\label{L2boundforhypertang}
\Vert E_Mf_{j,tang} \Vert^p_{HL^p(W\cap B_j)}\lesssim_{\varepsilon}K^{O(\varepsilon)}\rho^{p\varepsilon}\Vert f_{j,tang} \Vert^{p}_{L^2}
\end{equation}
holds for $p\geq \frac{2(n+2)}{n}$. Interpolating \eqref{L2energy} with \eqref{L2boundforhypertang}, we deduce that
\begin{equation}\label{interpolation}
\Vert E_Mf_{j,tang} \Vert^p_{HL^p(W\cap B_j)}\lesssim_{\varepsilon}\rho^{\frac{n+2}{2}-\frac{np}{4}+O(\varepsilon)}\Vert f_{j,tang} \Vert^p_{L^2}
\end{equation}
holds for $2\leq p \leq \frac{2(n+2)}{n}$. Applying \eqref{keytangentestimate} to the right-hand side of \eqref{interpolation}, we get
\begin{equation}\label{improvedtangBj}
\Vert E_Mf_{j,tang} \Vert^p_{HL^p(W\cap B_j)}
\lesssim_{\varepsilon}R^{\frac{n^2+2n-2}{2(n-1)}-\frac{(n+1)p}{4}+O(\varepsilon)}\Vert f_{j,tang} \Vert^{\frac{2n}{n-1}}_{L^2}\max_{d(\theta)=R^{-1/2}}\Vert f_{\theta} \Vert^{p-\frac{2n}{n-1}}_{L^{2}_{avg}(\theta)}.
\end{equation}
Note that $\Vert f_{j,tang} \Vert_{L^2}\lesssim \Vert f \Vert_{L^2}$. This together with \eqref{improvedtangBj} yields the desired bound for the tangential case whenever $p\geq p_{broad}$.

Combining the estimates in the cellular case, the transversal case and the tangential case, we conclude that Proposition \ref{strongermainprop0} holds.
\end{proof}

Now we show that Proposition \ref{strongermainprop0} implies Theorem \ref{localmainprop0}. By Proposition \ref{strongermainprop0}, we have
\begin{equation}\label{eq:hyperbroadputon}
\sum_{B\in \mathcal{B}}\min_{V_1,...,V_A} \max_{\tau \notin V_s}\Vert E_Mf_{\tau} \Vert^p_{L^p(B)}
\lesssim_{\varepsilon}R^{p\varepsilon}\Vert f \Vert^{\frac{2n}{n-1}}_{L^2}\Vert f \Vert^{p-\frac{2n}{n-1}}_{L^{\infty}}
\end{equation}
for $p\geq p_{broad}$, where $\tau \notin V_s$ is an abbreviation for $dist(\tau, V_s)>K_1^{-1},\;s=1,...,A$. By H\"{o}lder's inequality, we derive
\begin{equation}\label{eq:hyperbroadputon1}
\sum_{B\in \mathcal{B}}\min_{V_1,...,V_A} \max_{\tau \notin V_s}\Vert E_Mf_{\tau} \Vert^p_{L^p(B)}
\lesssim_{\varepsilon}R^{p\varepsilon}\Vert f \Vert^{2}_{L^2}\Vert f \Vert^{p-2}_{L^{\infty}},\;p\geq p_{broad}.
\end{equation}
For each $B\in \mathcal{B}$, we fix a choice of $V_1,...,V_A$ achieving the minimum above. Then, we can write
\[\Vert E_Mf \Vert^p_{L^p(B)}\lesssim K^{O(1)}\max_{\tau \notin V_s}\Vert E_Mf_{\tau} \Vert^p_{L^p(B)}+\sum^{A}_{s=1}\Vert \sum_{\alpha \in \mathcal{A}_s}E_Mf_{\alpha}\Vert^p_{L^p(B)},\]
where $\mathcal{A}_s$ denotes the minimal collection of $\alpha$ covering the $K_1^{-1}$-neighborhood of $V_s$. By Proposition \ref{prop:decoupling}, we then get
\[\Vert \sum_{\alpha \in \mathcal{A}_s}E_Mf_{\alpha}\Vert_{L^p(B)}\lesssim K_1^{m(\frac{1}{2}-\frac{1}{p})}(\log K_1)^{O(1)}\big(\sum_{\alpha \in \mathcal{A}_s} \|E_Mf_{\alpha}\|^{2}_{L^{p}(w_{B})}\big)^{\frac{1}{2}}.\]
By H\"{o}lder's inequality, we see that
\[\Vert \sum_{\alpha \in \mathcal{A}_s}E_Mf_{\alpha}\Vert^p_{L^p(B)}\lesssim K_1^{m(p-2)}(\log K_1)^{O(1)}\sum_{\alpha \in \mathcal{A}_s} \|E_Mf_{\alpha}\|^{p}_{L^{p}(w_{B})}.\]
In summary, we obtain
\[\Vert E_Mf \Vert^p_{L^p(B)}\lesssim K^{O(1)}\max_{\tau \notin V_s}\Vert E_Mf_{\tau} \Vert^p_{L^p(B)}+AK_1^{m(p-2)}(\log K_1)^{O(1)}\sum_{\alpha} \|E_Mf_{\alpha}\|^{p}_{L^{p}(w_{B})}.\]
Since
\[\sum_{B\in \mathcal{B}}w_{B}\lesssim w_{B_R},\]
we can sum over $B\in \mathcal{B}$ to conclude that
\begin{equation}\label{eq:nice}
\Vert E_Mf \Vert^p_{L^p(B_R)}\lesssim K^{O(1)}\Vert E_Mf \Vert^p_{HL^p(B_R)}
+AK_1^{m(p-2)}(\log K_1)^{O(1)}\sum_{\alpha}\|E_Mf_{\alpha}\|^{p}_{L^{p}(w_{B_R})}.
\end{equation}
By recaling and induction on scales, we obtain
\[\|E_Mf_{\alpha}\|^{p}_{L^{p}(w_{B_R})}\lesssim_{\varepsilon}R^{p\varepsilon}K_1^{-p\varepsilon}K_1^{2n-(n-1)p}\Vert f_{\alpha} \Vert^p_{L^p}.\]
Therefore, the second term on the right-hand side of \eqref{eq:nice} is
\[\lesssim AK_1^{2(n-m)-p(n-m-1)-p\varepsilon}(\log K_1)^{O(1)}\Vert f \Vert^p_{L^p}.\]
Recall that $K_1=R^{\varepsilon^4}$ and $A \sim R^{\varepsilon^6}$. The induction closes when $p\geq p_{narrow}$.
This together with \eqref{eq:hyperbroadputon1} gives
\[\Vert E_Mf \Vert_{L^p(B_R)}\lesssim_{\varepsilon}R^{O(\varepsilon)}\Vert f \Vert^{\frac{2}{p}}_{L^2}\Vert f \Vert^{1-\frac{2}{p}}_{L^{\infty}}\]
for $p\geq \bar{p}=\max\{p_{broad}, p_{narrow}\}$.

We now prove Theorem \ref{localmainprop0} by the argument on page 7 of \cite{Kim}. We deduce that the restricted stronger-type estimate
\begin{equation}\label{eq:restrictedstronger}
\Vert E_M\chi_F \Vert_{L^p(B_R)}\lesssim_{\varepsilon}R^{O(\varepsilon)}\vert F \vert^{1/p}
\end{equation}
holds for $p\geq \bar{p}$ and any set $F\subset B^{n-1}(0,1)$. Interpolating \eqref{eq:restrictedstronger} with the trivial $L^{\infty}-L^{\infty}$ estimate, we deduce that
\begin{equation}\label{eq:333}
\Vert E_Mf \Vert_{L^p(B_R)}\lesssim_{\varepsilon}R^{O(\varepsilon)}\Vert f \Vert_{L^p(B^{n-1}(0,1))}
\end{equation}
holds for $p> \bar{p}$ and any function $f\in L^p(B^{n-1}(0,1))$. By H\"{o}lder's inequality, we derive that \eqref{eq:localmainprop0} holds for $p\geq \bar{p}$, which completes the proof of Theorem \ref{localmainprop0}.

\vskip 0.12in

%\subsection*{Acknowledgements}

%\subsection*{Declarations}
%$\bullet$ Conflict of interest: There is no conflict of interest.

\begin{center}

\end{center}

\end{document}